\documentclass{amsart}

\usepackage[latin1]{inputenc}
\usepackage[T1]{fontenc}
\usepackage{verbatim}
\usepackage{multirow}

\usepackage{geometry}
\usepackage{amssymb}
\usepackage{amsmath}
\usepackage{graphicx}
\usepackage{amsthm}
\usepackage{caption}

\usepackage{color}

\usepackage{enumerate}

\newcommand{\eChar}{\begin{enumerate}[(i)]}
\newcommand{\eCharR}{\begin{enumerate}[(a)]}
\newcommand{\eBr}{\begin{enumerate}[(1)]}

\newcommand{\Abstract}

\title
{
Lin-Lu-Yau curvature and diameter of amply regular graphs
}

\author{Xintian Li}
\address{
School of Mathematical Sciences\\
University of Science and Technology of China\\
96 Jinzhai Road\\
Hefei 230026\\
Anhui Province\\
China}
\email{lxt024@mail.ustc.edu.cn}

\author{Shiping Liu}
\address{
School of Mathematical Sciences\\
University of Science and Technology of China\\
96 Jinzhai Road\\
Hefei 230026\\
Anhui Province\\
China}
\email{spliu@ustc.edu.cn}
\date{\today}

\theoremstyle{plain}
\newtheorem{lemma}{Lemma}[section]
\newtheorem{theorem}[lemma]{Theorem}

\newtheorem{corollary}[lemma]{Corollary}

\theoremstyle{definition}

\newtheorem{definition}[lemma]{Definition}
\newtheorem{remark}[lemma]{Remark}

\numberwithin{equation}{section}

\begin{document}

\pagestyle{plain}

\begin{abstract}
 By Hall's marriage theorem, we study lower bounds of the Lin-Lu-Yau curvature of amply regular graphs with girth $3$ or $4$ under different parameter restrictions. As a consequence, we show that each conference graph has positive Lin-Lu-Yau curvature. 
Our approach also provides a geometric proof of a classical diameter estimate for amply regular graphs in the case of girth $4$ and some special cases of girth $3$.
\end{abstract}
\maketitle
%%%%%%%%%%%%%%%%%%%%%%%%%%%%%%%%%%%%%%%%%%%%%%%%%%%%%
\section{Introduction and statements of result}
%%%%%%%%%%%%%%%%%%%%%%%%%%%%%%%%%%%%%%%%%%%%%%%%%%%%%
Ricci curvature is a fundamental concept in Riemannian geometry. Its extension to general metric measure spaces, particularly, to locally finite graphs, has attracted lots of attention \cite{CY96, Ollivier09, Ollivier10, BJL12, JL14, BM15, LLY11, LY10}. In 2009, Ollivier \cite{Ollivier09,Ollivier10} introduced the notion of coarse Ricci curvature of Markov chains on metric spaces including graphs. On graphs, Ollivier's Ricci curvature $\kappa_p$ of an edge is defined via the Wasserstein distance between two probability measure around the two vertices of the edge, depending on an idleness parameter $p\in [0,1]$. In 2011, Lin, Lu, and Yau \cite{LLY11} modified Ollivier's notion by taking the minus of the derivative of $\kappa_p$ at $p=1$. We will study this modified Ricci curvature, which will be refered to as the \emph{Lin-Lu-Yau curvature}, on amply regular graphs in this paper.

Let $G=(V,E)$ be a locally finite connected simple graph. Recall that the girth of $G$ is the length of its shortest cycle. We denote by $d(x,y)$ the length of the shortest path connecting the two vertices $x$ and $y$. We call $\mu: V\to [0,1]$ a probability measure on the graph $G=(V,E)$ if $\sum_{v\in V}\mu(v)=1$.
\begin{definition}
Let $G=(V,E)$ be a locally finite graph, $\mu_1$ and $\mu_2$ be two probability measures on $G$. The Wasserstein distance $W_1(\mu_1, \mu_2)$ between $\mu_1$ and $\mu_2$ is defined as
\[W_1(\mu_1,\mu_2)=\inf_{\pi}\sum_{y\in V}\sum_{x\in V}d(x,y)\pi(x,y),\]
where the infimum is taken over all maps $\pi: V\times V\to [0,1]$ satisfying
\[\mu_1(x)=\sum_{y\in V}\pi(x,y),\,\,\mu_2(y)=\sum_{x\in V}\pi(x,y).\] Such a map is called a transport plan.
\end{definition}
We consider the following particular measure around a vertex $x\in V$:
\[\mu_x^p(y)=\left\{
               \begin{array}{ll}
                 p, & \hbox{if $y=x$;} \\
                 \frac{1-p}{\mathrm{deg}(x)}, & \hbox{if $y\sim x$;} \\
                 0, & \hbox{otherwise,}
               \end{array}
             \right.
\]
where $\mathrm{deg}(x):=\sum_{y\in V: y\sim x} 1$ is the vertex degree of $x$.

\begin{definition}[\cite{Ollivier09,LLY11}] Let $G=(V,E)$ be a locally finite graph. For any vertices $x,y\in V$, the $p$-Ollivier-Ricci curvature $\kappa_p(x,y)$, $p\in [0,1]$, is defined as
\[\kappa_p(x,y)=1-\frac{W_1(\mu_x^p,\mu_y^p)}{d(x,y)}.\]
The Lin-Lu-Yau curvature $\kappa(x,y)$ is defined as
\[\kappa(x,y)=\lim_{p\to 1}\frac{\kappa_p(x,y)}{1-p}.\]
\end{definition}
Notice that $\kappa_1(x,y)$ is always $0$. Hence, Lin-Lu-Yau curvature $\kappa(x,y)$ is minus the derivative of $\kappa_p(x,y)$ at $p=1$.

Bourne et al. \cite{BCLMP18} studied the relation between $p$-Ollivier-Ricci curvature and Lin-Lu-Yau curvature. In particular, they proved for an edge $xy\in E$ with $\mathrm{deg}(x)=\mathrm{deg}(y)=d$ that
\begin{equation}\label{eq:idle}
 \kappa(x,y)=\frac{d+1}{d}\kappa_{\frac{1}{d+1}}(x,y).
\end{equation}

The Lin-Lu-Yau curvature has been computed or estimated on graphs with further regularity assumptions. For regular graphs (i.e., every vertex has the same degree), the following upper bound estimate is known.
\begin{theorem}[see \cite{CKKLMP20}]\label{Thm:upperT}
Let $G=(V,E)$ be a $d$-regular graph. For any edge $xy\in E$, we have
\[\kappa(x,y)\leq \frac{2+|\Delta_{xy}|}{d},\]
where $\Delta_{xy}:=\Gamma(x)\cap \Gamma(y)$, $\Gamma(x):=\{z\in V| xz\in E\}$, and $\Gamma(y):=\{z\in V| yz\in E\}$.
\end{theorem}

Bonini et al. \cite{BCDDFP20} derived Lin-Lu-Yau curvature formulas for strongly regular graphs in terms of the graph parameters and the size of a maximal matching in the so-called core neighborhood. In fact, a more general curvature formula for regular graphs has been shown in \cite[Theorem 2.6]{MW19}. In particular, their result leads to exact formula for the Lin-Lu-Yau curvature for strongly regular graphs with girth $4$ and $5$. For the case of girth $3$, no exact formula for the Lin-Lu-Yau curvature purely in terms of graph parameter exists: The $4\times 4$ Rook's graph and Shrikhande graph are both strongly regular with the parameter $(16,6,2,2)$; Bonini et al. \cite{BCDDFP20} computed their Lin-Lu-Yau curvature to be $\kappa=\frac{2}{3}$ and $\kappa=\frac{1}{3}$ respectively.

We study the Lin-Lu-Yau curvature of \emph{amply regular} graphs with girth $3$ or $4$ in this paper.
\begin{definition}[Amply regular graph \cite{BCN89}] We call a $d$-regular graph with $n$ vertices an amply regular graph with parameter $(n,d,\alpha,\beta)$ if the following holds true:
\begin{itemize}
  \item [(i)] Any two adjacent vertices have $\alpha$ common neighbors;
  \item [(ii)] Any two vertices with distance $2$ have $\beta$ common neighbors.
\end{itemize}
\end{definition}
We remark that if the above property (ii) holds for any two non-adjacent vertices, the amply regular graph is strongly regular. Therefore, amply regularity is a relaxation of the strongly regularity.

For amply regular graphs with girth $4$ we have the following Lin-Lu-Yau curvature formula.

\begin{theorem}
 Let $G=(V,E)$ be an amply regular graph with parameter $(n,d,\alpha,\beta)$ with girth $4$. For any $xy\in E$, we have
\[\kappa(x,y)=\frac{2}{d}.\]
\end{theorem}
This formula has been established for the particular cases of strongly regular graphs with girth $4$ and distance regular graphs with girth $4$ in \cite{BCDDFP20} and \cite{CKKLP20} respectively. Observe that the Lin-Lu-Yau curvature of a given edge only involves number of common neighbors of vertices with distance at most $2$. Therefore, the proofs of \cite{BCDDFP20} and \cite{CKKLP20} apply directly to amply regular graphs with girth $4$.

Our main result is the following Lin-Lu-Yau curvature formula or estimates for amply regular graphs with girth $3$ (i.e., $\alpha\geq 1$).

\begin{theorem}\label{Thm:main}
 Let $G=(V,E)$ be an amply regular graph with parameter $(n,d,\alpha,\beta)$.
\begin{itemize}
  \item [(i)]If $\alpha=1$ and $\alpha<\beta$, then we have for any $xy\in E$ that
\[\kappa(x,y)=\frac{3}{d};\]
  \item [(ii)] If $\alpha\geq 1$ and $\alpha=\beta-1$, then we have for any $xy\in E$ that \[\kappa(x,y)\geq \frac{2}{d};\]
  \item [(iii)] If $\alpha=\beta>1$, then we have for any $xy\in E$ that \[\kappa(x,y)\geq \frac{2}{d}.\]
\end{itemize}
\end{theorem}
\begin{remark}
 \begin{itemize}
   \item [(1)] Consider the $9$-Paley graph which is strongly regular with parameter $(9,4,1,2)$. It fulfills the parameter restrictions in Theorem \ref{Thm:main} (i) and (ii). We can check directly the Lin-Lu-Yau curvature of the $9$-Paley graph is $\frac{3}{4}$.
   \item [(2)] Consider the Shrikhande graph which is strongly regular with parameter $(16,6,2,2)$. It fulfills the parameter restriction in Theorem \ref{Thm:main} (iii). We can check directly the Lin-Lu-Yau curvature of the Shrikhande graph is $\frac{1}{3}$. Therefore, the estimate in Theorem \ref{Thm:main} (iii) is sharp.
   \item [(3)] Bonini et al. \cite[Conjecture 1.7]{BCDDFP20} conjectured that the Lin-Lu-Yau curvature of any strongly regular conference graphs with parameter $(4\gamma+1,2\gamma,\gamma-1,\gamma)$ with $\gamma\geq 2$ satisfies
\[\kappa(x,y)=\frac{1}{2}+\frac{1}{2\gamma}, \,\,\text{for all}\,\,xy\in E.\]
Our Theorem \ref{Thm:main} (ii) implies for such conference graphs that
\[\kappa(x,y)\geq \frac{1}{\gamma}, \,\,\text{for all}\,\,xy\in E.\]
\item [(4)] In general, it is still open whether the Lin-Lu-Yau curvature of an amply regular graph of parameter $(n,d,\alpha,\beta)$ with girth $3$ (i.e., $\alpha\geq 1$) and $\beta\geq 2$ is always nonnegative or not.
 \end{itemize}
\end{remark}

For Ollivier's Ricci curvature, a Bonnet-Myers type diameter estimate holds true \cite{Ollivier09}: Uniformly positive curvature lower bound implies the finiteness of the diameter. This has been extended to Lin-Lu-Yau curvature.

\begin{theorem}[Discrete Bonnet-Myers Theorem \cite{LLY11}]\label{Thm:DBMT} Let $G=(V,E)$ be a locally finite connected graph. Suppose $\kappa(x,y)\geq k>0$ holds true for any $xy\in E$. Then the diameter
\[\mathrm{diam}(G)\leq \frac{2}{k}.\]
\end{theorem}
Discrete Bonnet-Myers theorem has recently found important applications in coding theory: it provides a completely elementary way to derive bounds on locally correctable and some locally testable binary linear codes \cite{IS20}.

Applying Theorem \ref{Thm:DBMT}, we have the following consequences.
\begin{corollary}\label{coro:diam}
  Let $G=(V,E)$ be an amply regular graph with parameter $(n,d,\alpha,\beta)$.
\begin{itemize}
\item [(i)] If $G$ has girth $4$, then \[\mathrm{diam}(G)\leq d;\]
  \item [(ii)]If $\alpha=1$ and $\alpha<\beta$,  then \[\mathrm{diam}(G)\leq \frac{2d}{3};\]
  \item [(iii)] If $\alpha\geq 1$ and $\alpha=\beta-1$,  then \[\mathrm{diam}(G)\leq d;\]
  \item [(iv)] If $\alpha=\beta>1$, then \[\mathrm{diam}(G)\leq d.\]
\end{itemize}
\end{corollary}

The following diameter estimate for amply regular graphs is known via classical combinatorial methods, see, e.g., \cite[Theorem 1.13.2]{BCN89}.
\begin{theorem}\cite{BCN89}\label{thm:BCN}
Let $G=(V,E)$ be an amply regular graph with parameter $(n,d,\alpha,\beta)$. If $\alpha\leq \beta\neq 1$, then
 \[\mathrm{diam}(G)\leq d,\]
where the equality holds if and only if $G$ is a $d$-hypercube graph.
\end{theorem}
\begin{remark}
 Corollary \ref{coro:diam} provides a geometric proof via Lin-Lu-Yau curvature for Theorem \ref{thm:BCN} in the case of girth $4$ and some special cases of girth $3$. Moreover, we improve the estimate in Theorem \ref{thm:BCN} under the condition of Corollary \ref{coro:diam} (ii).
\end{remark}

%%%%%%%%%%%%%%%%%%%%%%%%%%%%%%%%%%%%%%%%%%%%%%%%%%%%%%
\section{Preliminaries}
%%%%%%%%%%%%%%%%%%%%%%%%%%%%%%%%%%%%%%%%%%%%%%%%%%%%%%
We first recall the important concept of matching from graph theory.
\begin{definition}\cite[Section 16.1]{BM08}
Let $G=(V,E)$ be a locally finite simple connected graph. A set $M$ of pairwise nonadjacent edges is called a \emph{matching}. The two vertices of each edge of $M$ are said to be \emph{matched} under $M$, and each vertex adjacent to an edge of $M$ is said to be \emph{covered} by $M$. A matching $M$ is called a \emph{perfect matching} if it covers every vertex of the graph.
\end{definition}
The following Hall's marriage theorem will be an important tool for our purpose.
\begin{lemma}\cite[Theorem 16.4]{BM08}\label{lemma:Hall}
Let $H=(V,E)$ be a bipartite graph with the bipartition $V=S\sqcup T$. Then $H$ has a matching which covers every vertex in $S$ if and only if
\[|\Gamma_T(W)|\geq |W| \,\,\text{for all}\,\,W\subseteq S\]
holds, where $\Gamma_T(W):=\{v\in T |\,\,\text{there exists }\,\, w\in W \,\,\text{such that}\,\, vw\in E\}$.
 \end{lemma}
For any $xy\in E$, the Lin-Lu-Yau curvature $\kappa(x,y)$ only depends on the subgraph induced by vertices with distance less than or equal to $2$ to $x$ and $y$ \cite[Lemma 2.3]{BM15}. For convenience, we introduce the following notation of the core neighborhood of $xy\in E$:
\[C_{xy}=\{x\}\cup\{y\}\cup\Delta_{xy}\cup N_x\cup N_y\cup P_{xy},\]
where $N_x=\Gamma(x)\setminus(\{y\}\cup\Gamma(y))$, $N_y=\Gamma(y)\setminus(\{x\}\cup\Gamma(x))$, $P_{xy}=\{z\in V | d(x,z)=2, d(y,z)=2\}$.

%%%%%%%%%%%%%%%%%%%%%%%%%%%%%%%%%%%%%%%%%%%%%%%%%%%%%%
\section{Proof of Theorem \ref{Thm:main}}
%%%%%%%%%%%%%%%%%%%%%%%%%%%%%%%%%%%%%%%%%%%%%%%%%%%%%%
In this section, we prove our main Theorem \ref{Thm:main}.
\begin{proof}[Proof of Theorem \ref{Thm:main}(i)]
We consider the core neighborhood decomposition of an edge $xy\in E$, that is,  $$\Gamma(x)=\{y\}\cup \Delta_{xy}\cup N_x,\,\,\Gamma(y)=\{x\}\cup \Delta_{xy}\cup N_y.$$
Since $\alpha=1$, we can denote $\Delta_{xy}=\{x_0\}$, and there are no edges connecting $x_0$ and any vertex in $N_x$ or $N_y$. We are going to show the existence of a perfect matching between $N_x$ and $N_y$ via applying Lemma \ref{lemma:Hall}.

Let $H$ be the bipartite subgraph with vertex set $V_H:=N_x\sqcup N_y$ and edge set $E_H:=\{vw\in E|v\in N_x, w\in N_y\}$. Take a subset $A\subset N_x$, let $B:=\Gamma_{N_y}(A)$ be the set of neighbors of $A$ in $N_y$. Observe that for any vertex $x_i\in A$, we have $d(x_i,y)=2$. Therefore, there are $\beta$ common neighbors of $x_i$ and $y$. Since $x$ is a common neighbor of $x_i$ and $y$, there are $\beta-1$ neighbors of $x_i$ in $N_y$. Similarly, for any vertex $y_i\in B\subset N_y$, there are $\beta-1$ neighbors of $y_i$ in $N_x$.
Denote by $E(A,B):=\{xy\in E | x\in A, y\in B\}$. We then have
\[\sum_{v\in A}(\beta-1)=|E(A,B)|\leq \sum_{w\in B}(\beta-1).\]
Since $\beta>1$, we derive from above that \[|A|\leq |B|=|\Gamma_{N_y}(A)|.\]
Applying Lemma \ref{lemma:Hall}, there is a perfect matching $M$ between $N_x$ and $N_y$. We construct the following transport plan building upon such a perfect matching:
\[\pi(v,w):=\left\{
              \begin{array}{ll}
                0, & \hbox{$v=x,w=y$;} \\
                \frac{1}{d+1}, & \hbox{$v=w\in \Delta_{xy}\cup\{x\}\cup\{y\}$;} \\
                \frac{1}{d+1}, & \hbox{$vw\in M$;} \\
                0, & \hbox{otherwise.}
              \end{array}
            \right.
\]
Noticing that $|N_x|=|N_y|=d-\alpha-1$, we calculate the Wasserstein distance
\[W_1(\mu_x^{\frac{1}{d+1}},\mu_y^{\frac{1}{d+1}})\leq\sum_{v\in V}\sum_{w\in V}d(v,w)\pi(v,w)=\frac{|N_x|}{d+1}=\frac{d-\alpha-1}{d+1}=\frac{d-2}{d+1}.\]
Applying (\ref{eq:idle}), we have the Lin-Lu-Yau curvature
\[\kappa(x,y)=\frac{d+1}{d}\kappa_{\frac{1}{d+1}}(x,y)=\frac{d+1}{d}(1-W_1(\mu_x^{\frac{1}{d+1}},\mu_y^{\frac{1}{d+1}}))\geq\frac{3}{d}.\]
Combining with Theorem \ref{Thm:upperT}, we obtain
\[\kappa(x,y)=\frac{3}{d}.\]
\end{proof}
\begin{proof}[Proof of Theorem \ref{Thm:main}(ii)]
We construct a bipartite graph $H$ from the core neighborhood of $xy\in E$ as follows. Let us denote
\[\Delta_{xy}=\{z_1,\ldots,z_\alpha\}.\]
We add a new set of vertices $\Delta'_{xy}:=\{z_1',\ldots,z_\alpha'\}$ which is considered as a copy of $\Delta_{xy}$. Let $H$ be the bipartite graph with vertex set
\[V_H=N_x\cup\Delta_{xy}\cup\Delta_{xy}'\cup N_y\]
and edge set
\[E_H=E_1\cup E_2\cup E_3\cup E_4\cup E_5,\]
where
\begin{align*}
 E_1=&\{vw | v\in N_x, w\in N_y, vw\in E\},\\
E_2=&\{vz_i' | v\in N_x, z_i'\in \Delta_{xy}', vz_i\in E\},\\
E_3=&\{z_iw | z_i\in \Delta_{xy}, w\in N_y, z_iw\in E\},\\
E_4=&\{z_iz_i' | i=1,\ldots,\alpha\},\\
E_5=&\{z_iz_j' | z_iz_j\in E, i\neq j, 1\leq i,j\leq \alpha\}.
\end{align*}
Notice that in the above construction, edges only exist between $N_x\cup \Delta_{xy}$ and $N_y\cup \Delta_{xy}'$. We will show that $H$ has a perfect matching by Lemma \ref{lemma:Hall}.

Take a subset $A\subset N_x\cup \Delta_{xy}$. Let $B=\Gamma_{N_y\cup \Delta_{xy}'}(A)$ be the neighbors of $A$ in $N_y\cup\Delta_{xy}'$.
\begin{itemize}
  \item The case $A\subset N_x$. Similarly as in the proof of Theorem \ref{Thm:main} (i), there are $\beta-1\geq 1$ neighbors of any $v\in A$ in $N_y\cup \Delta_{xy}'$, and $\beta-1\geq 1$ neighbors of any $w\in B\cap N_y$ in $N_x\cup \Delta_{xy}$. Consider any vertex $z_i'\in B\cap \Delta_{xy}'$. The corresponding vertex $z_i\in \Delta_{xy}$ and $x$ have $\alpha$ common neighbors in $G$ including the vertex $y$. Since there is a new additional edge between $z_i$ and $z_i'$ in $H$, the vertex $z_i'$ has exactly $\alpha$ neighbors in $N_x\cup \Delta_{xy}$ in $H$.
Therefore, we derive
\[\sum_{v\in A}(\beta-1)=|E(A,B)|\leq \sum_{w\in B\cap N_y}(\beta-1)+\sum_{w\in B\cap\Delta_{xy}'}\alpha.\]
By assumption, we have $\alpha=\beta-1$. Hence the above estimate yields
\[|A|\leq |B|=|\Gamma_{N_y\cup\Delta_{xy}'}(A)|.\]
  \item The case $A\subset \Delta_{xy}$.  Consider any vertex $z_i\in A$. The vertices $z_i$ and $y$ has $\alpha$ common neighbors in $G$ including the vertex $x$. Since there is a new additional edge between $z_i$ and $z_i'\in \Delta_{xy}'$ in $H$, the vertex $z_i$ has exactly $\alpha$ neighbors in $N_y\cup \Delta_{xy}'$ in $H$. Therefore, we derive
\[\sum_{v\in A}\alpha=|E(A,B)|\leq\sum_{w\in B\cap N_y}(\beta-1)+\sum_{w\in B\cap \Delta_{xy}'}\alpha.\]
Using $\alpha=\beta-1$, we obtain
\[|A|\leq |B|=|\Gamma_{N_y\cup\Delta_{xy}'}(A)|.\]
  \item The case $A\subset N_x\cup \Delta_{xy}$. We have
\[\sum_{v\in A\cap N_x}(\beta-1)+\sum_{v\in A\cap \Delta_{xy}}\alpha=|E(A,B)|\leq \sum_{w\in B\cap N_y}(\beta-1)+\sum_{w\in B\cap \Delta_{xy}'}\alpha.\]
By $\alpha=\beta-1$, we obtain
\[|A|\leq |B|=|\Gamma_{N_y\cup\Delta_{xy}'}(A)|.\]
\end{itemize}
In conclusion, we have $|A|\leq |B|=|\Gamma_{N_y\cup\Delta_{xy}'}(A)|$ for any subset $A\subset N_x\cup \Delta_{xy}$. Then, Lemma \ref{lemma:Hall} implies the existence of a perfect matching of $H$. Therefore, there exists a transport plan between $\mu_x^{\frac{1}{d+1}}$ and $\mu_y^{\frac{1}{d+1}}$ in which the mass at any vertex in $N_x\cup \Delta_{xy}$ is moved by a distance at most $1$. And the mass at $\{x\}$ and $\{y\}$ stay put. Using such a transport plan, we estimate the Wasserstein distance
\[W_1(\mu_x^{\frac{1}{d+1}},\mu_y^{\frac{1}{d+1}})\leq\sum_{v\in V}\sum_{w\in V}d(v,w)\pi(v,w)\leq \frac{d-1}{d+1}.\]
Applying (\ref{eq:idle}), we have the Lin-Lu-Yau curvature
\[\kappa(x,y)=\frac{d+1}{d}\kappa_{\frac{1}{d+1}}(x,y)=\frac{d+1}{d}(1-W_1(\mu_x^{\frac{1}{d+1}},\mu_y^{\frac{1}{d+1}}))\geq\frac{2}{d}.\]
\end{proof}
\begin{proof}[Proof of Theorem \ref{Thm:main}(iii)] We modify the construction of the bipartite graph $H$ in the proof of Theorem \ref{Thm:main} (ii) by dropping the edge set $E_4$. That is, $H$ is the graph with vertex set $V_H=N_x\cup \Delta_{xy}\cup\Delta_{xy}'\cup N_y$ and edge set $E_H=E_1\cup E_2\cup E_3\cup E_5$.

Take a subset $A\subset N_x\cup \Delta_{xy}$. Let $B=\Gamma_{N_y\cup \Delta_{xy}'}(A)$ be the neighbors of $A$ in $N_y\cup\Delta_{xy}'$. Similarly as the analysis in the proof of Theorem \ref{Thm:main} (ii), we have
\[\sum_{v\in A\cap N_x}(\beta-1)+\sum_{v\in A\cap \Delta_{xy}}(\alpha-1)=|E(A,B)|\leq \sum_{w\in B\cap N_y}(\beta-1)+\sum_{w\in B\cap \Delta_{xy}'}(\alpha-1).\]
By assumption, we have $\alpha=\beta>1$. Then we derive from above that
\[|A|\leq |B|=|\Gamma_{N_y\cup\Delta_{xy}'}(A)|.\]
Therefore, there is a perfect matching of $H$ by Lemma \ref{lemma:Hall}. Similarly as in the proof of Theorem \ref{Thm:main} (ii), we derive
\[\kappa(x,y)\geq \frac{2}{d}.\]
\end{proof}
\section*{Acknowledgement}
We are very grateful to Shuliang Bai for pointing out that amply regular graphs of parameter $(n,d,\alpha,\beta)$ with $\beta=1$ can have girth $3$ and negative Lin-Lu-Yau curvature.
We would like to thank the anonymous referees for suggestions that helped to greatly improved the quality of this paper.
This work is supported by the National Natural Science Foundation of China (No. 12031017).

\end{document}